\documentclass[final,a4paper]{siamltex}
\usepackage{amsmath,amssymb,amsfonts}
\usepackage{color}
\usepackage{enumerate}
\usepackage{verbatim}
\usepackage{eepic,epic,epsfig,epstopdf}
\usepackage[]{graphics}
\usepackage{graphicx,inputenc}
\usepackage{subfigure}
\usepackage{changepage}
\usepackage{amsmath}
\usepackage{threeparttable}
\usepackage{multirow}
\usepackage{extarrows}
\usepackage{url,color}
\usepackage{graphicx,subfig}
\graphicspath{{./pics/}}
\usepackage{bm,color,url}
\usepackage{algorithm,algorithmic}
\usepackage{geometry}
\usepackage{setspace}
\usepackage{booktabs}
\usepackage[notref,notcite]{showkeys}
\newtheorem{assumption}{Assumption}[section]
\newtheorem{remark}{Remark}[section]

\allowdisplaybreaks[4]

\def\RelU2{\mathrm{ReLU}^2}
\def\RelU3{\mathrm{ReLU}^3}

\title{A rate of convergence
of Physics Informed Neural Networks for the linear second order elliptic PDEs}
\author{
	Yuling Jiao\thanks{School of Mathematics and Statistics, and
		Hubei Key Laboratory of Computational Science, Wuhan University, P.R. China. (yulingjiaomath@whu.edu.cn)}\quad\and
	Yanming Lai \thanks{School of Mathematics and Statistics,  Wuhan University, P.R. China. (laiyanming@whu.edu.cn)}\quad\and
	 Dingwei Li\thanks{School of Mathematics and Statistics, Wuhan University, P.R. China.
  ({lidingv@whu.edu.cn})}\quad
\and Xiliang Lu\thanks{School of Mathematics and Statistics,  and
		Hubei Key Laboratory of Computational Science, Wuhan University, P.R. China.
  ({xllv.math@whu.edu.cn})}
\quad\and
Fengru Wang \thanks{School of Mathematics and Statistics,  and
		Hubei Key Laboratory of Computational Science, Wuhan University, P.R. China.
 ({wangfr@whu.edu.cn})}
\quad\and
 Yang Wang \thanks{Department of Mathematics, The Hong Kong University of Science and Technology,
Clear Water Bay, Kowloon, Hong Kong ({yangwang@ust.hk})}
\quad \and
 Jerry Zhijian Yang\thanks{School of Mathematics and Statistics, and
		Hubei Key Laboratory of Computational Science, Wuhan University, P.R. China.
  ({zjyang.math@whu.edu.cn})}
}

\begin{document}
	\maketitle
	
	\begin{abstract}
In recent years, physical informed neural networks (PINNs) have been  shown  to be a powerful tool for solving PDEs empirically. However, numerical analysis of PINNs  is still missing. In this paper, we prove the convergence rate to PINNs for the second order elliptic equations with Dirichlet boundary condition, by establishing the upper bounds on the number of training samples, depth and width of the deep neural networks to achieve desired accuracy. The error of PINNs is decomposed into approximation error and statistical error, where the approximation error is given in $C^2$ norm with $\mathrm{ReLU}^{3}$ networks (deep network with activations function $\max\{0,x^3\}$) and the statistical error is estimated by Rademacher complexity. We derive the bound on the Rademacher complexity of the non-Lipschitz composition of gradient norm with $\mathrm{ReLU}^{3}$ network, which is of immense independent interest.
\end{abstract}

\begin{keywords}
PINNs, $\mathrm{ReLU}^3$ neural network, B-splines,  Rademacher complexity.
\end{keywords}


\section{Introduction}
\label{section intro}
\par Classical  numerical methods such as the finite element method are successful to solve the low-dimensional PDEs, see e.g., \cite{brenner2007mathematical,ciarlet2002finite,Quarteroni2008Numerical,Thomas2013Numerical,Hughes2012the}. However these methods may encounter some difficulties in both theoretical analysis and numerical implementation for the high-dimensional PDEs.
 Motivated by the facts that deep learning method for  high-dimensional data analysis has been achieved great successful applications in discriminative, generative and reinforcement learning \cite{he2015delving,Goodfellow2014Generative,silver2016mastering}, solving high dimensional PDEs with  deep neural networks becomes an extremely potential  approach and  has  attracted a lot of attentions  \cite{Cosmin2019Artificial,Justin2018DGM,DeepXDE,raissi2019physics,Weinan2017The,Yaohua2020weak,Berner2020Numerically,Han2018solving}.
 Due to the excellent approximation ability of the deep neural networks, several numerical schemes  have been proposed to solve PDEs with deep neural networks  including the deep Ritz method (DRM)  \cite{Weinan2017The}, physics-informed neural networks (PINNs) \cite{raissi2019physics}, and deep Galerkin method (DGM) \cite{Yaohua2020weak}.
 Both DRM  and DGM are applied to variational forms of PDEs, and PINNs are based on residual minimization to the differential equation, see \cite{Cosmin2019Artificial,Justin2018DGM,DeepXDE,raissi2019physics}, which  can be extended to general PDEs \cite{jagtap2020conservative,lagaris1998artificial,fpinns,npinns}.
	
Despite the above mentioned deep PDEs solvers work well empirically, rigorous numerical analysis for these methods are far from complete.
The convergence rate of DRM with two layer networks and deep networks are studied in \cite{luo2020two,hong2021rademacher,lu2021priori,duan2021convergence}, the convergence of PINNs are given in
\cite{shin2020convergence,shin2020error,mishra2020}.
In this work, we will provide the nonasymptotic convergence rate of the PINNs with $\mathrm{ReLU}^3$ networks, i.e., a quantitative error estimation with respect to the topological structure of the neural networks (the depth and width) and the number of the samples. Hence it gives a rule to determine the hyper-parameters to achieve a desired accuracy. Our contributions are summarized as follows.

\noindent\textbf{Our contributions and main results}
	\begin{itemize}
		\item
We obtain the approximation results   $\mathrm{ReLU}^3$ network in $C^2(\bar\Omega)$, see Theorem  \ref{app error}, i.e.,
		 $\forall \bar u \in C^3(\bar\Omega) $ and for any $\epsilon>0$, there exists a $\mathrm{ReLU}^3$ network $u_\phi$ with depth $\lceil\log_2d\rceil+2$ and width $C(d,\|\bar u\|_{C^3(\bar{\Omega})})\left(\frac{1}{\epsilon}\right)^d$ such that
		\begin{equation*}
			\|\bar u-u_\phi\|_{C^2(\bar{\Omega})}\leq \epsilon,
		\end{equation*}
where $d,\ \bar{\Omega},\ C(d,\|\bar{u}\|)$ stands for the dimension of $x$, the closure of the domain $\Omega$ and some numerical  constant that only depends on  $(d,\|\bar{u}\|)$, respectively.
		\item
We establish an upper  bound of  the statistical error for PINNs by applying the tools of Pseudo dimension, especially we give an upper bound of
 the Rademacher complexity  to the derivative of  $\mathrm{ReLU}^3$ network  which is non-Lipschitz composition with   $\mathrm{ReLU}^3$ network,  via  calculating  the
Pseudo dimension of networks with  $\mathrm{ReLU}$, $\mathrm{ReLU}^2$ and  $\mathrm{ReLU}^3$ activation functions, see Theorem \ref{sta error}.
We prove that $\forall \mathcal{D},\mathcal{W}\in\mathbb{N}$ and  $\epsilon>0$, if the number of training samples
in PINNs is with the order $O\left(\mathcal{D}^6\mathcal{W}^2(\mathcal{D}+\log\mathcal{W})\left(\frac{1}{\epsilon}\right)^{2+\delta}\right)$,
		where $\delta$ is an arbitrarily positive number, then
		the statistical error \begin{equation*}
			\mathbb{E}_{\{{X_k}\}_{k=1}^{N},\{{Y_k}\}_{k=1}^{M}}\sup_{u\in\mathcal{P}}\left|\mathcal{L}(u)-\widehat{\mathcal{L}}(u)\right|\leq \epsilon,
		\end{equation*}
where $\mathcal{L}$ and   $\widehat{\mathcal{L}}$  are loss functions defined in  (\ref{lossp}) and (\ref{losss}) respectively.
\item Based on the above  bounds on approximation error and statistical error, we establish a   nonasymptotic convergence rate to PINNs for second order elliptic equations  with Dirichlet boundary condition, see Theorem \ref{total error}, i.e.,
	$\forall \epsilon >0$ if we set the depth and width by $$\mathcal{D} = \mathcal{O}(\lceil\log_2d\rceil+2),\quad \mathcal{W} = \mathcal{O}\left(\frac{1}{\epsilon^d}\right)$$ in the  $\mathrm{ReLU}^3$ network
		and set the number of training  samples used in PINNs as
		$\mathcal{O}(\left(\frac{1}{\epsilon}\right)^{2d+4+\delta}),$
		where $\delta$ is an arbitrarily positive number, then
		\begin{equation*}
			\mathbb{E}_{\{{X_k}\}_{k=1}^{N},\{{Y_k}\}_{k=1}^{M}}\left\|\widehat{u}_{\phi}-u^{*}\right\|_{H^{\frac{1}{2}}(\Omega)}\leq\epsilon,
		\end{equation*}
	\end{itemize}
where $\widehat{u}_{\phi}$ is the minimizer in (\ref{eqn:pinns}).

	The paper is organized as follows. In Section \ref{section the pinns method and error decomposition} we describe the problem setting and introduce the error decomposition for the PINNs. In Section \ref{section approximation error} the approximation error of  $\mathrm{ReLU}^3$ networks in Sobolev spaces based on approximation results of  B-splines is proved. In Section \ref{section statistical error} we provide the statistical error estimation by using the bound of Rademacher complexity. In Section \ref{section convergence rate for the pinns} the convergence rate of PINNs is shown.  A conclusion and short discussion is given in Section \ref{conclusion}.

\section{The PINNs Method and Error Decomposition}\label{section the pinns method and error decomposition}

\subsection{Preliminary and PINNS method}\label{notation}
We give the notations of neural networks and function spaces which will be used later. Let $\mathcal{D}\in\mathbb{N}$, a function $f$ is called a neural network if it is implemented by:
\begin{equation*}
\begin{array}{l}
\mathbf{f}_{0}(\mathbf{x})=\mathbf{x},\\
\mathbf{f}_{\ell}(\mathbf{x})=\mathbf{\varrho}_{\ell}\left(A_{\ell} \mathbf{f}_{\ell-1}+\mathbf{b}_{\ell}\right)
=(\mathbf{\varrho}_{i}^{(\ell)}(\left(A_{\ell} \mathbf{f}_{\ell-1}+\mathbf{b}_{\ell}\right)_i))
\quad \text { for } \ell=1, \ldots, \mathcal{D}-1, \\
\mathbf{f}:=\mathbf{f}_{\mathcal{D}}(\mathbf{x})=A_{\mathcal{D}}\mathbf{f}_{\mathcal{D}-1}+\mathbf{b}_{\mathcal{D}},
\end{array}
\end{equation*}
where $A_{\ell}=\left(a_{ij}^{(\ell)}\right)\in\mathbb{R}^{n_{\ell}\times n_{\ell-1}}$,
$\mathbf{b}_{\ell}=\left(b_i^{(\ell)}\right)\in\mathbb{R}^{n_{\ell}}$ and $\mathbf{\varrho}_{\ell}:\ \mathbb{R}^{n_{\ell}}\rightarrow\mathbb{R}^{n_{\ell}}$ is the active function. The hyper-parameters $\mathcal{D}$ and $\mathcal{W}:=\max\{N_{\ell},\ell=0,\cdots,\mathcal{D}\}$ are called the depth and the width of the network, respectively. Also $\sum_{\ell=1}^{L} n_{\ell} $ is called the number of units of $\mathbf{f}$ and  $\phi = \{A_{\ell},\mathbf{b}_{\ell}\}_{\ell}$ are called the weight parameters. Let $\Phi$ be a set of activation functions and $X$ be a Banach space, we define the neural network function class by
\begin{equation*}
\begin{aligned}
\mathcal{N}(\mathcal{D},\mathcal{W},\{\|\cdot\|_{X},\mathcal{B}\},\Phi):
=\{&f:f \text{ is implemented by a neural network with depth }\mathcal{D}\\
&\text{ and width } \mathcal{W}, \|f\|_X\leq\mathcal{B}, \text{ and }\varrho_{i}^{(\ell)}\in\Phi \text{ for each }i \text{ and }\ell\}.
\end{aligned}
\end{equation*}

Then we introduce the function spaces and the partial differential equations what we are interested in. The governing equation is defined in an open bounded domain $\Omega\subset\mathbb{R}^d$ with smooth boundary $\partial\Omega$, and without loss of generality we may assume that $\Omega\subset[0,1]^d$. Let $\alpha=(\alpha_1,\cdots,\alpha_n)$ be an $n$-dimensional index with $|\alpha|:=\sum_{i=1}^{n}\alpha_i$ and $s$ be a nonnegative integer. The standard function spaces include continuous function space, $L^p$ space, Sobolev spaces are given below.
	\begin{align*}
	&C(\Omega):=\{\text{all the continuous functions defined on }\Omega\},\quad C^s(\Omega):=\{f:\Omega\to\mathbb{R}\ |\ D^{\alpha}f\in C(\Omega)\},\\
	&C(\bar\Omega):=\{\text{all the continuous functions defined on }\bar\Omega\},\quad \|f\|_{C(\bar\Omega)}:=\max_{x\in\bar{\Omega}}|f(x)|,\\
	&C^s(\bar\Omega):=\{f:\bar\Omega\to\mathbb{R}\ |\ D^{\alpha}f\in C(\bar\Omega)\},\quad \|f\|_{C^s(\bar\Omega)}:=\max_{x\in\bar{\Omega},|\alpha|\leq s}|D^{\alpha}f(x)|,\\
	&L^p(\Omega):=\left\{f:\Omega\to\mathbb{R}\ |\ \int_{\Omega}|f|^pdx<\infty\right\},\quad \|f\|_{L^p(\Omega)}:=\left[\int_{\Omega}|f|^p(x)dx\right]^{1/p},\quad \forall p\in[1,\infty),\\
	&L^{\infty}(\Omega):=\{f:\Omega\to\mathbb{R}\ |\ \exists C>0 \ s.t.\ |f|\leq C \ a.e.\},\quad\|f\|_{L^{\infty}(\Omega)}:=\inf\{C \ | \ |f|\leq C \ a.e.\}, \\
    & W^{s,p}(\Omega):=\{f:\Omega\to\mathbb{R}\ |\ D^{\alpha}f\in L^p(\Omega),\quad |\alpha|\leq s\},\quad\|f\|_{W^{s,p}(\Omega)}:=\left(\sum_{|\alpha|\leq s}\|D^{\alpha}f\|_{L^p(\Omega)}^p\right)^{1/p}
\end{align*}
If $s$ is a nonnegative real number, the fractional Sobolev space $W^{s,p}(\Omega)$ can be defined as follows: setting $\theta=s-\lfloor{s}\rfloor$ and
\begin{align*}
&W^{s,p}(\Omega):=\left\{f:\Omega\to\mathbb{R}\ |\ \int_{\Omega}\int_{\Omega}\frac{|D^{\alpha}f(x)-D^{\alpha}f(y)|^p}{|x-y|^{\theta p+d}}dxdy<\infty,\quad\forall |\alpha|=\lfloor{s}\rfloor\right\},\\
&\|f\|_{W^{s,p}(\Omega)}:=\left(\|f\|_{W^{\lfloor{s}\rfloor,p}(\Omega)}^p+\sum_{|\alpha|=\lfloor{s}\rfloor}\int_{\Omega}\int_{\Omega}\frac{|D^{\alpha}f(x)-D^{\alpha}f(y)|^p}{|x-y|^{\theta p+d}}dxdy\right)^{1/p}.
\end{align*}
Let $C_{0}^{\infty }({\Omega} )$ be the set of smooth functions with compact support in $\Omega$, and $W_0^{s,p}(\Omega)$ is the completion space of $C_0^{\infty}(\Omega)$ in $W^{s,p}(\Omega)$. For $s<0$, $W^{s,p}(\Omega)$ is the dual space of $W_{0}^{-s,q}(\Omega)$ with $q$ satisfying $\frac{1}{p}+\frac{1}{q}=1$. When $p=2$, $W^{s,p}(\Omega)$ is a Hilbert space and it is also denoted by $H^s(\Omega)$. The constant $C$ may vary from place to place but it is independent to the hyper-parameters of neural networks.

We will consider the following linear second order elliptic equation with Dirichlet boundary condition, where $\frac{\partial u}{\partial x_{i}}(x)$ and $\frac{\partial^{2} u}{\partial x_{i} \partial x_{j}}(x)$ are denoted by $u_{x_{i}}$ and $u_{x_{i} x_{j}}$, respectively.
	\begin{equation} \label{second order elliptic equation}
		\left\{\begin{aligned}
			-\sum_{i,j=1}^{d}a_{ij}u_{x_ix_j}+\sum_{i=1}^{d}b_iu_{x_i}+cu&=f  \quad\text { in } \Omega, \\
			eu &= g   \quad\text { on } \partial \Omega, \\
		\end{aligned}\right.
	\end{equation}
	where $a_{ij}\in C(\bar{\Omega})$, $b_i,c\in L^{\infty}(\Omega)$, $f\in L^2(\Omega)$, $e\in C(\partial\Omega)$, $g\in L^2(\partial\Omega)$ with the strictly elliptic condition, i.e., there exists a constant $\lambda> 0$ such that $\sum_{i,j}a_{ij}\xi_i\xi_j \geq \lambda|\xi|^2,\quad\forall x\in\Omega,\xi\in\mathbb{R}^d$. Define the constants by $\mathfrak{A}=\max_{i,j}\{\|a_{ij}\|_{C(\bar{\Omega})}\}$, $\mathfrak{B}=\max_{i}\{\|b_{i}\|_{L^{\infty}(\Omega)}\}$, $\mathfrak{C}=\|c\|_{L^{\infty}(\Omega)}$, $\mathfrak{E}=\|e\|_{C(\partial\Omega)}$, $\mathfrak{F}=\|f\|_{L^2(\Omega)}$, $\mathfrak{G}=\|g\|_{L^2(\partial\Omega)}$, and $\mathfrak{Z}=\{\mathfrak{A},\mathfrak{B},\mathfrak{C},\mathfrak{E},\mathfrak{F},\mathfrak{G}\}$.


Instead of solving problem \eqref{second order elliptic equation}, we consider a minimization problem with the loss functional $\mathcal{L}$ on $C^2(\Omega)\cap C(\bar\Omega)$:
	\begin{equation*}
		\mathcal{L}(u):=\int_{\Omega}\left(-\sum_{i,j=1}^{d}a_{ij}u_{x_ix_j}+\sum_{i=1}^{d}b_iu_{x_i}+cu-f\right)^2dx
		+\int_{\partial\Omega}\left(eu - g\right)^2dx.
	\end{equation*}

\begin{assumption}\label{ass:exist}
Assume that (\ref{second order elliptic equation}) has a unique strong solution $u^*\in C^2({\Omega})\cap C(\bar\Omega)$.
\end{assumption}

If Assumption \ref{ass:exist} holds, then $u^*$ is also the unique minimizer of loss functional $\mathcal{L}$. Define $|\Omega|$ and $|\partial\Omega|$ be the measure of $\Omega$ and its boundary respectively, i.e., $|\Omega|:=\int_{\Omega}1 dx$ and $|\partial\Omega|:=\int_{\partial\Omega}1 ds$, then $\mathcal{L}$ can be equivalently written by
\begin{align}
\mathcal{L}(u)=&|\Omega|\mathbb{E}_{X\sim U(\Omega)}\left(-\sum_{i,j=1}^{d}a_{ij}(X)u_{x_ix_j}(X)+\sum_{i=1}^{d}b_i(X)u_{x_i}(X)+c(X)u(X)-f(X)\right)^2  \nonumber\\
&+|\partial\Omega|\mathbb{E}_{Y\sim U(\partial\Omega)}\left(e(Y)u(Y) - g(Y)\right)^2\label{lossp},
\end{align}
	where $U(\Omega)$ and $U(\partial\Omega)$ are uniform distribution on $\Omega$ and $\partial\Omega$, respectively.

To solve minimization of $\mathcal{L}(u)$ numerically, a standard discrete version of $\mathcal{L}$ is given by:
	\begin{align}
	\widehat{\mathcal{L}}(u)=&\frac{|\Omega|}{N}\sum_{k=1}^{N}\left(-\sum_{i,j=1}^{d}a_{ij}(X_k)u_{x_ix_j}(X_k)+\sum_{i=1}^{d}b_i(X_k)u_{x_i}(X_k)+c(X_k)u(X_k)-f(X_k)\right)^2\nonumber \\
	&+\frac{|\partial\Omega|}{M}\sum_{k=1}^{M}\left(e(Y_k)u(Y_k) - g(Y_k)\right)^2, \label{losss}
	\end{align}
	where $\{X_k\}_{k=1}^{N}$ and $\{Y_k\}_{k=1}^{M}$ are i.i.d. random samples according to the uniform distribution  $U(\Omega)$ on $\Omega$ and $U(\partial\Omega)$ on $\partial \Omega$, respectively.
	
The PINNs method considers a minimization problem with respect to $\widehat{\mathcal{L}}$:
\begin{equation}\label{eqn:pinns}
\min_{u_{\phi}\in\mathcal{P}}\widehat{\mathcal{L}}(u_{\phi}),
\end{equation}
where the admissible set $\mathcal{P}$ refers to the deep neural network function class parameterized by $\phi$.
Assume there exists at least one minimizer to \eqref{eqn:pinns}, which is denote by $\widehat{u}_{\phi}$, and we will give an error estimation to $u^*$ and $\widehat{u}_\phi$ for some carefully chosen admissible set $\mathcal{P}$. The definition of $\mathcal{P}$ will be given at beginning of Section \ref{section approximation error}.

\begin{remark}
In practical, it is difficult to compute $\widehat{u}_\phi$ precisely due to the nonlinear structure of $\mathcal{P}$. People often call a (random) solver $\mathcal{A}$, say SGD, to
solve \eqref{eqn:pinns} and let the output $u_{\phi\mathcal{A}}$ be the final solution. In this work we ignore the optimization error and only consider the error between $u^*$ and $\widehat{u}_\phi$.
\end{remark}

\begin{remark}
The existence of the minimizer to problem \eqref{eqn:pinns} depends on the choice of the admissible set $\mathcal{P}$. In case of the minimizer may not exist for some admissible set $\mathcal{P}$, we may consider a $\gamma$-optimal solution $\widehat{u}_{\phi}$, i.e., $\widehat{\mathcal{L}}(\widehat{u}_{\phi}) \leq \inf_{u_{\phi}\in\mathcal{P}}\widehat{\mathcal{L}}(u_{\phi}) + \gamma$ for a arbitrary small positive $\gamma$. The main results in this work are also true for the $\gamma$-optimal solution.
\end{remark}

\subsection{Error Decomposition}	
	The a priori estimation to equation (\ref{second order elliptic equation}) can be find in \cite{agmon1959estimates}, and we list it below for completeness.
	\begin{lemma}\cite{agmon1959estimates} \label{Lp estimate}
		For $u\in H^{\frac{1}{2}}(\Omega)\bigcap L^2(\partial\Omega)$,
		\begin{equation*}
			\|{u}\|_{H^{\frac{1}{2}}(\Omega)}\leq
			C\left \|-\sum_{i,j=1}^{d}a_{ij}{u}_{x_ix_j}+\sum_{i=1}^{d}b_i{u}_{x_i}+c{u}\right\|_{H^{-\frac{3}{2}}(\Omega)}+C\left\|e{u}\right\|_{L^2(\partial\Omega)}.
		\end{equation*}
	\end{lemma}
Next we decompose the error into the approximation error and statistical error separately.
	\begin{proposition} \label{error decomposition}
		Assume that $\mathcal{P}\subset H^2(\Omega)\bigcap C(\bar{\Omega})$, then
		\begin{align*}
			&\|\widehat{u}_{\phi}-u^*\|_{H^{\frac{1}{2}}(\Omega)}^2\\
			&\leq \underbrace{C\inf_{\bar{u}\in\mathcal{P}}\left[3\max\{2d^2\mathfrak{A}^2,d\mathfrak{B}^2,\mathfrak{C}^2\}\|\bar{u}-u^*\|_{H^2(\Omega)}^2+|\partial\Omega|\mathfrak{E}^2\|\bar{u}-u^*\|_{C(\partial{\Omega})}^2\right]}_{\mathcal{E}_{app}}\\
			&\quad+\underbrace{C \sup _{u \in \mathcal{P}}\left|\mathcal{L}(u)-\widehat{\mathcal{L}}(u)\right|}_{\mathcal{E}_{sta}}.
		\end{align*}
	\end{proposition}
\begin{proof}
		For any $\bar{u}\in\mathcal{P}$, we have
		\begin{equation*}
			\begin{split}
				\mathcal{L}\left(\widehat{u}_{\phi}\right)-\mathcal{L}\left(u^{*}\right) &	= \widehat{\mathcal{L}}\left(\widehat{u}_{\phi}\right)-\widehat{\mathcal{L}}\left(\bar{u}\right) +\widehat{\mathcal{L}}\left(\bar{u}\right)-\mathcal{L}\left(\bar{u}\right)+\mathcal{L}\left(\bar{u}\right)-\mathcal{L}\left(u^{*}\right) \\
				&\leq  \left[\mathcal{L}\left(\bar{u}\right)-\mathcal{L}\left(u^{*}\right)\right]+\sup _{u \in \mathcal{P}}\left|\mathcal{L}(u)-\widehat{\mathcal{L}}(u)\right|,
			\end{split}
		\end{equation*}
where the last step is due to the fact that $\widehat{\mathcal{L}}\left(\widehat{u}_{\phi}\right)-\widehat{\mathcal{L}}\left(\bar{u}\right)\leq 0$. Since $\bar{u}$ can be any element in $\mathcal{P}$, we take the infimum for both side and obtain:
		\begin{equation} \label{error decomposition1}
							\mathcal{L}\left(\widehat{u}_{\phi}\right)-\mathcal{L}\left(u^{*}\right)
				\leq \inf_{\bar{u}\in\mathcal{P}}\left[\mathcal{L}\left(\bar{u}\right)-\mathcal{L}\left(u^{*}\right)\right]+ \sup _{u \in \mathcal{P}}\left|\mathcal{L}(u)-\widehat{\mathcal{L}}(u)\right|.
		\end{equation}
		Since $u^*$ is the strong solution of equation $(\ref{second order elliptic equation})$ and $\mathcal{L}(u^*)=0$, then $\forall u\in\mathcal{P}$ we have
		\begin{equation} \label{error decomposition2}
			\begin{split}
				&\mathcal{L}({u})-\mathcal{L}(u^*)
				=\int_{\Omega}\left(-\sum_{i,j=1}^{d}a_{ij}u_{x_ix_j}+\sum_{i=1}^{d}b_iu_{x_i}+cu-f\right)^2dx+\int_{\partial\Omega}(eu-g)^2dx\\	
				=&\int_{\Omega}\left(-\sum_{i,j=1}^{d}a_{ij}({u}-u^*)_{x_ix_j}+\sum_{i=1}^{d}b_i({u}-u^*)_{x_i}+c({u}-u^*)\right)^2dx+\int_{\partial\Omega}e^2({u}-u^*)^2dx\\
				\leq& 3\int_{\Omega}\left[\left(\sum_{i,j=1}^{d}a_{ij}({u}-u^*)_{x_ix_j}\right)^2+\left(\sum_{i=1}^{d}b_i({u}-u^*)_{x_i}\right)^2+\left(c({u}-u^*)\right)^2\right]dx\\
				&\quad+\int_{\partial\Omega}e^2({u}-u^*)^2dx\\
				\leq & 3\max\{2d^2\mathfrak{A}^2,d\mathfrak{B}^2,\mathfrak{C}^2\}\|{u}-u^*\|_{H^2(\Omega)}^2+|\partial\Omega|\mathfrak{E}^2\|{u}-u^*\|_{C({\partial{\Omega}})}^2.
			\end{split}
		\end{equation}
		On the other hand, by Lemma  \ref{Lp estimate} and inequality $\|\cdot\|_{H^{-\frac{3}{2}}(\Omega)}\leq C\|\cdot\|_{L^2(\Omega)}$, we have $\forall u \in H^{\frac{1}{2}}(\Omega)\bigcap L^2(\partial\Omega)$
		\begin{equation} \label{error decomposition3}
			\begin{split}
				&\mathcal{L}({u})-\mathcal{L}(u^*)\\	
				= & \int_{\Omega}\left(-\sum_{i,j=1}^{d}a_{ij}({u}-u^*)_{x_ix_j}+\sum_{i=1}^{d}b_i({u}-u^*)_{x_i}+c({u}-u^*)\right)^2dx+\int_{\partial\Omega}e^2({u}-u^*)^2dx\\
				\geq & C\|{u}-u^*\|_{H^{\frac{1}{2}}(\Omega)}^2.
			\end{split}
		\end{equation}
		Combining (\ref{error decomposition1})-(\ref{error decomposition3}) yields the result.	
	\end{proof}

The approximation  error $\mathcal{E}_{app}$  describes the expressive power of the parameterized function class $\mathcal{P}$ in $H^2(\Omega)$ and $C(\bar{\Omega})$ norm, which  corresponds to the approximation error   in  FEM  known as the C$\acute{\text{e}}$a's lemma \cite{ciarlet2002finite}.
      The statistical error $\mathcal{E}_{sta}$  is caused by the Monte Carlo discretization of $\mathcal{L}(\cdot)$ defined in \ref{lossp} with $\widehat{\mathcal{L}}(\cdot)$ in \ref{losss}.

	\section{Approximation Error}\label{section approximation error}
We will choose $\mathcal{P}$ as a $\mathrm{ReLU}^3$ networks, to ensure $\mathcal{P}\subset H^2(\Omega)\bigcap C(\bar{\Omega})$. More precisely,
$$\mathcal{P}=\mathcal{N}(\mathcal{D},\mathcal{W},\{\|\cdot\|_{C^2(\bar{\Omega})},\mathcal{B}\},\{\mathrm{ReLU}^3\}),$$
where the hyper-parameters $\mathcal{B} = 2\|u^*\|_{C^2(\bar{\Omega})}$ and $\{\mathcal{D},\mathcal{W}$\} will  be given in later discussions (i.e., Theorem \ref{total error}) to ensure the desired accuracy. The $\mathrm{ReLU}^3$ network is refer to a neural network where the active function $\sigma(x)$ is given by
\begin{equation}\label{eqn:relu3}		
\sigma (x)=\left\{\begin{array}{ll}
				x^3, & x\geq 0, \\
				0, & \text{others}.
			\end{array}\right.
\end{equation}
It can be seen that $\sigma(x)$ is twice differentiable.

\begin{theorem} \label{app error}
Given $\bar u \in C^3(\bar\Omega)$ and for any $\epsilon>0$, there exists a $\mathrm{ReLU}^3$ network $u$ with depth $\lceil\log_2d\rceil+2$ and width $C(d,\|\bar u\|_{C^3(\bar{\Omega})})\left(\frac{1}{\epsilon}\right)^d$ such that
$
\|\bar u-u\|_{C^2(\bar{\Omega})}\leq \epsilon.
$
\end{theorem}
	\begin{proof}
A special case in \cite{appreluk}.
	\end{proof}

	\section{Statistical Error}\label{section statistical error}

In this section we give the statistical error with  parameterized function class $\mathcal{P}$, by establishing the Rademacher complexity of the non-Lipschitz composition of $\mathrm{ReLU}^3$ network and its partial derivative. The technique used here may be helpful to analysis the statistical errors for other deep PDEs solvers, where the main difficulties are also to estimated the
Rademacher complexity of non-Lipschitz composition induced by the derivative operator.

Firstly by carefully computation and triangle inequality, we have the following Lemma.
\begin{lemma} \label{L decomposition}
		\begin{equation*}
			\mathbb{E}_{\{{X_k}\}_{k=1}^{N},\{{Y_k}\}_{k=1}^{M}}\sup_{u\in \mathcal{P}}\left|\mathcal{L}(u)-\widehat{\mathcal{L}}(u)\right|	\leq\sum_{j=1}^{13}\mathbb{E}_{\{{X_k}\}_{k=1}^{N},\{{Y_k}\}_{k=1}^{M}}\sup_{u\in \mathcal{P}}\left|\mathcal{L}_j(u)-\widehat{\mathcal{L}_j}(u)\right|
		\end{equation*}
		where
{\small
$$\mathcal{L}_1(u)=|\Omega|\mathbb{E}_{X\sim U(\Omega)}\left(\sum_{i,j=1}^{d}a_{ij}(X)u_{x_ix_j}(X)\right)^2,  \ \ \mathcal{L}_2(u)=|\Omega|\mathbb{E}_{X\sim U(\Omega)}\left(\sum_{i=1}^{d}b_i(X)u_{x_i}(X)\right)^2,$$
$$\mathcal{L}_3(u)=|\Omega|\mathbb{E}_{X\sim U(\Omega)}\left(c(X)u(X)\right)^2, \ \ \mathcal{L}_4(u)=|\Omega|\mathbb{E}_{X\sim U(\Omega)}f(X)^2,$$
  $${\mathcal{L}_5}(u)=-2|\Omega|\mathbb{E}_{X\sim U(\Omega)}\left(\sum_{i,j=1}^{d}a_{ij}(X)u_{x_ix_j}(X)\right)\left(\sum_{i=1}^{d}b_i(X)u_{x_i}(X)\right),$$
		$${\mathcal{L}_6}(u)=-2|\Omega|\mathbb{E}_{X\sim U(\Omega)}\left(\sum_{i,j=1}^{d}a_{ij}(X)u_{x_ix_j}(X)\right)\cdot c(X)u(X),$$
			$${\mathcal{L}_7}(u)=2|\Omega|\mathbb{E}_{X\sim U(\Omega)}\left(\sum_{i,j=1}^{d}a_{ij}(X)u_{x_ix_j}(X)\right)\cdot f(X),$$ $${\mathcal{L}_8}(u)=2|\Omega|\mathbb{E}_{X\sim U(\Omega)}\left(\sum_{i=1}^{d}b_i(X)u_{x_i}(X)\right)\cdot c(X)u(X),$$
            $${\mathcal{L}_9}(u)=-2|\Omega|\mathbb{E}_{X\sim U(\Omega)}\left(\sum_{i=1}^{d}b_i(X)u_{x_i}(X)\right)\cdot f(X),$$ $${\mathcal{L}_{10}}(u)=-2|\Omega|\mathbb{E}_{X\sim U(\Omega)}c(X)u(X)f(X), \ \ \mathcal{L}_{11}(u)=|\partial\Omega|\mathbb{E}_{Y\sim U(\partial\Omega)}\left(e(Y)u(Y)\right)^2,$$
			$$\mathcal{L}_{12}(u)=|\partial\Omega|\mathbb{E}_{Y\sim U(\partial\Omega)}\left(g(Y)\right)^2, \ \ {\mathcal{L}_{13}}(u)=-2|\partial\Omega|\mathbb{E}_{Y\sim U(\partial\Omega)}e(Y)u(Y)g(Y),$$}
and $\widehat{\mathcal{L}}_j(u)$ is the discrete sample  version of $\mathcal{L}_j(u)$ by replacing expectation with sample average $j=1,...13$.
\end{lemma}

\subsection{Rademacher complexity, Covering Number and Pseudo-dimension}
	By the technique of symmetrization, we can bound the difference between continuous loss $\mathcal{L}_i$ and empirical loss $\widehat{\mathcal{{L}}}_i$ via Rademacher complexity.
	\begin{definition}\cite{wellner}
		The Rademacher complexity of a set $A \subseteq \mathbb{R}^N$ is defined by
		\begin{equation*}
			\mathfrak{R}(A) = \mathbb{E}_{\{\sigma_i\}_{k=1}^N}\left[\sup_{a\in A}\frac{1}{N}\sum_{k=1}^{N} \sigma_k a_k\right]
		\end{equation*}
		where,   $\{\sigma_k\}_{k=1}^N$ are $N$ i.i.d  Rademacher variables with $\mathbb{P}(\sigma_k = 1) = \mathbb{P}(\sigma_k = -1) = \frac{1}{2}.$
		
		Let $\Omega$ be a set and $\mathcal{F}$ be a function class which maps $\Omega$ to $\mathbb{R}$. Let $P$ be a probability distribution over $\Omega$ and $\{X_k\}_{k=1}^{N}$ be i.i.d. samples from $P$. The Rademacher complexity of $\mathcal{F}$ associated with distribution $P$ and sample size $N$ is defined  by
		\begin{equation*}
			\mathfrak{R}_{P,N}(\mathcal{F}) = \mathbb{E}_{\{X_k,\sigma_k\}_{k=1}^{N}}\left[\sup_{u\in \mathcal{F}}\frac{1}{N}\sum_{k=1}^N \sigma_k u(X_k)\right].
		\end{equation*}
	\end{definition}

For Rademacher complexity, we have following structural result.
\begin{lemma} \label{structural result of Rademacher}
Let $\Omega$ be a set and $P$ be a probability distribution over $\Omega$. Let $N\in\mathbb{N}$. Assume that $w:\Omega\to\mathbb{R}$ and $|w(x)|\leq\mathcal{B}$ for all $x\in\Omega$, then for any function class $\mathcal{F}$ mapping $\Omega$ to $\mathbb{R}$, there holds
\begin{equation*}
\mathfrak{R}_{P,N}(w(x)\mathcal{F})\leq\mathcal{B}\mathfrak{R}_{P,N}(\mathcal{F}),
\end{equation*}	
where
$
w(x)\mathcal{F}:=\{\bar{u}:\bar{u}(x)=w(x)u(x),u\in\mathcal{F}\}.$
\end{lemma}

\begin{proof}
{\small \begin{equation*}
\begin{aligned}
			&\mathfrak{R}_{P,N}(w(x)\cdot\mathcal{F})
			=\frac{1}{N}\mathbb{E}_{\{X_k,\sigma_k\}_{k=1}^{N}}\sup_{u\in\mathcal{F}}\sum_{k=1}^{N}\sigma_kw(X_k)u(X_k)\\
			&=\frac{1}{2N}\mathbb{E}_{\{X_k\}_{k=1}^{N}}\mathbb{E}_{\{\sigma_k\}_{k=2}^{N}}\sup_{u\in\mathcal{F}}\left[w(X_1)u(X_1)+\sum_{k=2}^{N}\sigma_kw(X_k)u(X_k)\right]\\
			&\quad\ +\frac{1}{2N}\mathbb{E}_{\{X_k\}_{k=1}^{N}}\mathbb{E}_{\{\sigma_k\}_{k=2}^{N}}\sup_{u\in\mathcal{F}}\left[-w(X_1)u(X_1)+\sum_{k=2}^{N}\sigma_kw(X_k)u(X_k)\right]\\
			&=\frac{1}{2N}\mathbb{E}_{\{X_k\}_{k=1}^{N}}\mathbb{E}_{\{\sigma_k\}_{k=2}^{N}}\\
			&\quad\sup_{u,u'\in\mathcal{F}}\left[w(X_1)[u(X_1)-u'(X_1)]+\sum_{k=2}^{N}\sigma_kw(X_k)u(X_k)+\sum_{k=2}^{N}\sigma_kw(X_k)u'(X_k)\right]\\
			&\leq\frac{1}{2N}\mathbb{E}_{\{X_k\}_{k=1}^{N}}\mathbb{E}_{\{\sigma_k\}_{k=2}^{N}}\\
			&\quad\sup_{u,u'\in\mathcal{F}}\left[\mathcal{B}|u(X_1)-u'(X_1)|+\sum_{k=2}^{N}\sigma_kw(X_k)u(X_k)+\sum_{k=2}^{N}\sigma_kw(X_k)u'(X_k)\right]\\
			&=\frac{1}{2N}\mathbb{E}_{\{X_k\}_{k=1}^{N}}\mathbb{E}_{\{\sigma_k\}_{k=2}^{N}}\\
			&\quad\sup_{u,u'\in\mathcal{F}}\left[\mathcal{B}[u(X_1)-u'(X_1)]+\sum_{k=2}^{N}\sigma_kw(X_k)u(X_k)+\sum_{k=2}^{N}\sigma_kw(X_k)u'(X_k)\right]\\
			&=\frac{1}{N}\mathbb{E}_{\{X_k,\sigma_k\}_{k=1}^{N}}\sup_{u\in\mathcal{F}}\left[\sigma_1\mathcal{B}u(X_1)+\sum_{k=2}^{N}\sigma_kw(X_k)u(X_k)\right]\\
			&\leq\cdots\leq\frac{\mathcal{B}}{N}\mathbb{E}_{\{X_k,\sigma_k\}_{k=1}^{N}}\sup_{u\in\mathcal{F}}\sum_{k=1}^{N}\sigma_ku(X_k)=\mathcal{B}\mathfrak{R}_{P,N}(\mathcal{F})
\end{aligned}
\end{equation*}}
\end{proof}

\begin{lemma} \label{sym}
Let $\{X_k\}_{k=1}^N$ be i.i.d. samples from $U(\Omega)$, then we have
\begin{equation*}
			\mathbb{E}_{\{{X_k}\}_{k=1}^{N}}\sup_{u\in{\mathcal{P}}}\left|\mathcal{L}_j(u)-\widehat{\mathcal{L}_j}(u)\right|
			\leq C_d \mathfrak{Z}\mathfrak{R}_{U(\Omega),N}(\mathcal{F}_j), \ \ j= 1,...,13,
\end{equation*}
where
$$
\mathcal{F}_1=\{{\pm} f:\Omega\to\mathbb{R}\ |\ \exists {u\in{\mathcal{P}}} \ \ \mathrm{and} \ \ 1\leq i,j,i',j'\leq d \quad \mathrm{s.t.}\quad
f(x)=u_{x_ix_j}(x)u_{x_{i'}x_{j'}}(x)\}, $$
$$ \mathcal{F}_2=\{{\pm} f:\Omega\to\mathbb{R}\ |\ \exists {u\in{\mathcal{P}}} \  \  \mathrm{ and } \ \ 1\leq i,i'\leq d \quad \mathrm{s.t.} \quad
f(x)=u_{x_i}(x)u_{x_{i'}}(x)\},$$
$$\mathcal{F}_3=\{{\pm} f:\Omega\to\mathbb{R}\ |\ \exists {u\in{\mathcal{P}}} \quad \mathrm{s.t.} \quad
f(x)=u(x)^2\}, \ \ \mathcal{F}_4=\{{\pm} f:\Omega\to\mathbb{R}\ |\ -1,0,1\},$$
$$\mathcal{F}_5=\{{\pm} f:\Omega\to\mathbb{R}\ |\ \exists {u\in{\mathcal{P}}} \ \ \mathrm{ and } \ \ 1\leq i,j,i'\leq d \quad \mathrm{s.t.}\quad
f(x)=u_{x_ix_j}(x)u_{x_{i'}}(x)\},$$
$$\mathcal{F}_6=\{{\pm} f:\Omega\to\mathbb{R}\ |\ \exists {u\in{\mathcal{P}}} \ \ \mathrm{ and } \ \ 1\leq i,j\leq d \quad \mathrm{s.t.} \quad
f(x)=u_{x_ix_j}(x)u(x)\},$$
$$\mathcal{F}_7=\{{\pm} f:\Omega\to\mathbb{R}\ |\ \exists {u\in{\mathcal{P}}} \ \ \mathrm{ and }\ \  1\leq i,j\leq d \quad \mathrm{s.t.} \quad
f(x)=u_{x_ix_j}(x)\}$$
$$\mathcal{F}_8=\{{\pm} f:\Omega\to\mathbb{R}\ |\ \exists {u\in{\mathcal{P}}} \ \ \mathrm{ and }\ \  1\leq i\leq d \quad  \mathrm{s.t.} \quad
f(x)=u_{x_i}(x)u(x)\}$$
$$\mathcal{F}_9=\{{\pm} f:\Omega\to\mathbb{R}\ |\ \exists {u\in{\mathcal{P}}} \ \  \mathrm{ and } \ \  1\leq i\leq d \quad \mathrm{s.t.} \quad
f(x)=u_{x_i}(x)\},$$
$$ \mathcal{F}_{10}=\{{\pm f}:\Omega\to\mathbb{R}\ |\ \exists u\in\mathcal{P} \quad \mathrm{s.t.} \quad f(x)=u(x)\}, $$ $$\mathcal{F}_{11}={\{{\pm f}:\partial\Omega\to\mathbb{R}\ |\ \exists {u\in{\mathcal{P}}} \quad \mathrm{s.t.} \quad
f=u^2|_{\partial\Omega}\}},$$
$$\mathcal{F}_{12}=\{{\pm} f:\partial\Omega\to\mathbb{R}\ |\ -1,0,1\}, $$
$$\mathcal{F}_{13}={\{\pm f:\partial\Omega\to\mathbb{R}\ |\ \exists u\in\mathcal{P}\quad\mathrm{s.t.}\quad f=u|_{\partial\Omega}\}}$$


\end{lemma}
\begin{proof}
We only give the proof of $\mathbb{E}_{\{{X_k}\}_{k=1}^{N}}\sup_{u\in{\mathcal{P}}}\left|\mathcal{L}_1(u)-\widehat{\mathcal{L}_1}(u)\right|\leq4{|\Omega|}\mathfrak{A}^2d^4\mathfrak{R}(\mathcal{F}_1)$ since other inequalities can be shown similarly. We take $\{\widetilde{X_k}\}_{k=1}^{N}$ as an independent copy of $\{{X_k}\}_{k=1}^{N}$, then
{\small
\begin{align*}
&\left|\mathcal{L}_1(u)-\widehat{\mathcal{L}_1}(u)\right|=|\Omega|\left|\mathbb{E}_{X\sim U(\Omega)}\left(\sum_{i,j=1}^{d}a_{ij}(X)u_{x_ix_j}(X)\right)^2
-\frac{1}{N}\sum_{k=1}^{N}\left(\sum_{i,j=1}^{d}a_{ij}(X_k)u_{x_ix_j}(X_k)\right)^2\right|\\		&=|\Omega|\left|\mathbb{E}_{\{\widetilde{X_k}\}_{k=1}^{N}}\frac{1}{N}\sum_{k=1}^{N}\left(\sum_{i,j=1}^{d}a_{ij}(\widetilde{X_k})u_{x_ix_j}(\widetilde{X_k})\right)^2
-\frac{1}{N}\sum_{k=1}^{N}\left(\sum_{i,j=1}^{d}a_{ij}(X_k)u_{x_ix_j}(X_k)\right)^2\right|\\
&\leq{\frac{|\Omega|}{N}}\mathbb{E}_{\{\widetilde{X_k}\}_{k=1}^{N}}\sum_{i,i',j,j'=1}^{d} |\sum_{k=1}^{N}a_{ij}(\widetilde{X_k})a_{i'j'}(\widetilde{X_k})u_{x_ix_j}(\widetilde{X_k})u_{x_{i'}x_{j'}}(\widetilde{X_k})\\
&-a_{ij}({X_k})a_{i'j'}({X_k})u_{x_ix_j}({X_k})u_{x_{i'}x_{j'}}({X_k})|.
\end{align*}}
Hence
{\small
\begin{align*}
&\mathbb{E}_{\{{X_k}\}_{k=1}^{N}}\sup_{u\in{\mathcal{P}}}\left|\mathcal{L}_1(u)-\widehat{\mathcal{L}_1}(u)\right|\\
&\leq{\frac{|\Omega|}{N}}\mathbb{E}_{\{{X_k}\}_{k=1}^{N}}\sup_{u\in{\mathcal{P}}}\mathbb{E}_{\{\widetilde{X_k}\}_{k=1}^{N}}\sum_{i,i',j,j'=1}^{d}\\
&\quad\ \left|\sum_{k=1}^{N}a_{ij}(\widetilde{X_k})a_{i'j'}(\widetilde{X_k})u_{x_ix_j}(\widetilde{X_k})u_{x_{i'}x_{j'}}(\widetilde{X_k})-a_{ij}({X_k})a_{i'j'}({X_k})u_{x_ix_j}({X_k})u_{x_{i'}x_{j'}}({X_k})\right|\\
&\leq{\frac{|\Omega|}{N}}\mathbb{E}_{\{{X_k}\}_{k=1}^{N}}\mathbb{E}_{\{\widetilde{X_k}\}_{k=1}^{N}}\sup_{u\in{\mathcal{P}}}\sum_{i,i',j,j'=1}^{d}\\
&\quad\ \left|\sum_{k=1}^{N}a_{ij}(\widetilde{X_k})a_{i'j'}(\widetilde{X_k})u_{x_ix_j}(\widetilde{X_k})u_{x_{i'}x_{j'}}(\widetilde{X_k})-a_{ij}({X_k})a_{i'j'}({X_k})u_{x_ix_j}({X_k})u_{x_{i'}x_{j'}}({X_k})\right|\\
&={\frac{|\Omega|}{N}}\mathbb{E}_{\{{X_k}\}_{k=1}^{N}}\mathbb{E}_{\{\widetilde{X_k}\}_{k=1}^{N}}\sup_{u\in{\mathcal{P}}}\sum_{i,i',j,j'=1}^{d}\mathbb{E}_{\{{\sigma_k}\}_{k=1}^{N}}\\
&\quad\ \left|\sum_{k=1}^{N}\sigma_k\left[a_{ij}(\widetilde{X_k})a_{i'j'}(\widetilde{X_k})u_{x_ix_j}(\widetilde{X_k})u_{x_{i'}x_{j'}}(\widetilde{X_k})-a_{ij}({X_k})a_{i'j'}({X_k})u_{x_ix_j}({X_k})u_{x_{i'}x_{j'}}({X_k})\right]\right|\\
&\leq{\frac{|\Omega|}{N}}\mathbb{E}_{\{{X_k},\widetilde{X_k},{\sigma_k}\}_{k=1}^{N}}\sup_{u\in{\mathcal{P}}}\sum_{i,i',j,j'=1}^{d}\left|\sum_{k=1}^{N}\sigma_k\left[a_{ij}(\widetilde{X_k})a_{i'j'}(\widetilde{X_k})u_{x_ix_j}(\widetilde{X_k})u_{x_{i'}x_{j'}}(\widetilde{X_k})\right]\right|\\
&\quad+{\frac{|\Omega|}{N}}\mathbb{E}_{\{{X_k},\widetilde{X_k},{\sigma_k}\}_{k=1}^{N}}\sup_{u\in{\mathcal{P}}}\sum_{i,i',j,j'=1}^{d}\left|\sum_{k=1}^{N}\sigma_k\left[a_{ij}({X_k})a_{i'j'}({X_k})u_{x_ix_j}({X_k})u_{x_{i'}x_{j'}}({X_k})\right]\right|\\
&={\frac{2|\Omega|}{N}}\mathbb{E}_{\{{X_k},{\sigma_k}\}_{k=1}^{N}}\sup_{u\in{\mathcal{P}}}\sum_{i,i',j,j'=1}^{d}\left|\sum_{k=1}^{N}\sigma_k\left[a_{ij}({X_k})a_{i'j'}({X_k})u_{x_ix_j}({X_k})u_{x_{i'}x_{j'}}({X_k})\right]\right|\\
&\leq{\frac{2|\Omega|}{N}}\sum_{i,i',j,j'=1}^{d}\mathbb{E}_{\{{X_k},{\sigma_k}\}_{k=1}^{N}}\sup_{u\in{\mathcal{P}}}\left|\sum_{k=1}^{N}\sigma_k\left[a_{ij}({X_k})a_{i'j'}({X_k})u_{x_ix_j}({X_k})u_{x_{i'}x_{j'}}({X_k})\right]\right|\\
&={\frac{2|\Omega|}{N}}\sum_{i,i',j,j'=1}^{d}\mathbb{E}_{\{{X_k},{\sigma_k}\}_{k=1}^{N}}\sup_{{f_{i,i',j,j'}\in\mathcal{F}_1}}\sum_{k=1}^{N}\sigma_k\left[a_{ij}({X_k})a_{i'j'}({X_k}){f_{i,i',j,j'}({X_k})}\right]\\
&\leq{\frac{2|\Omega|\mathfrak{A}^2}{N}}\sum_{i,i',j,j'=1}^{d}\mathbb{E}_{\{{X_k},{\sigma_k}\}_{k=1}^{N}}\sup_{{f_{i,i',j,j'}\in\mathcal{F}_1}}\sum_{k=1}^{N}\sigma_k{f_{i,i',j,j'}({X_k})}\\
&\leq{2}|\Omega|\mathfrak{A}^2d^4\mathfrak{R}_{{U(\Omega),N}}(\mathcal{F}_1)	
\end{align*}}
where the second and eighth steps are from Jensen's inequality and Lemma \ref{structural result of Rademacher}, the third and seventh steps are due to the facts that the insertion of Rademacher variables doesn't change the distribution and that {$\mathcal{F}_1$ is symmetric(i.e., if $f\in\mathcal{F}_1$, then $-f\in\mathcal{F}_1$)}, respectively.
\end{proof}

Next we give an upper bound of Rademacher complexity in terms of the covering number of the corresponding function class.
\begin{definition}\cite{anthony2009neural}
Suppose that $W\subset\mathbb{R}^n$. For any $\epsilon>0$, let $V\subset\mathbb{R}^n$ be an $\epsilon$-cover of $W$ with respect to the distance  $d_{\infty}$, that is, for any $w\in W$, there exists a $v\in V$ such that $d_{\infty}(u,v)<\epsilon$, where $d_{\infty}$ is defined by
$d_{\infty}(u,v):={\max_{1\leq i\leq n}|u_i-v_i|}$. The covering number $\mathcal{C}(\epsilon,W,d_{\infty})$ is defined to be the minimum cardinality among all $\epsilon$-cover of $W$ with respect to the distance  $d_{\infty}$.
\end{definition}
	
\begin{definition}\cite{anthony2009neural}
Suppose that $\mathcal{F}$ is a class of functions from $\Omega$ to $\mathbb{R}$. Given $n$  sample  $\mathbf{Z}_n=(Z_1,Z_2,\cdots,Z_n) \in \Omega^n$, $\mathcal{F}|_{\mathbf{Z}_n}\subset\mathbb{R}^n$ is defined by
\begin{equation*}
\mathcal{F}|_{\mathbf{Z}_n}= \{(u(Z_1),u(Z_2),\cdots, u(Z_n)): u\in\mathcal{N}^{3}\}.
\end{equation*}
The uniform covering number $\mathcal{C}_{\infty}(\epsilon,\mathcal{F},n)$ is defined by
\begin{equation*}
\mathcal{C}_{\infty}(\epsilon,\mathcal{F},n)=\max_{\mathbf{Z}_n\in\Omega^n}\mathcal{C}(\epsilon,\mathcal{F}|_{\mathbf{Z}_n},d_{\infty}).
\end{equation*}
\end{definition}
	
\begin{lemma} \label{chaining}
{Let $\Omega$ be a set and $P$ be a probability distribution over $\Omega$. Let $N\in\mathbb{N}_{\geq1}$.} Let $\mathcal{F}$ be a class of functions from $\Omega$ to $\mathbb{R}$ such that $0\in \mathcal{F}$ and the diameter of $\mathcal{F}$ is less than $\mathcal{B}$, i.e., $\|u\|_{L^\infty(\Omega)} \leq \mathcal{B}, \forall u \in \mathcal{F}$.  Then
$$\mathfrak{R}_{P,N}(\mathcal{F}) \leq  \inf_{0<\delta<\mathcal{B}}\left(4 \delta+\frac{12}{\sqrt{N}} \int_{\delta}^{\mathcal{B}} \sqrt{\log (2\mathcal{C}_{\infty}\left( \epsilon, \mathcal{F}, N\right))}\mathrm{d}\epsilon\right).$$
\end{lemma}
\begin{proof}
The proof is based on the chaining method, see  \cite{wellner}.
\end{proof}
	
By Lemma \ref{chaining}, we have to
bound the  covering number, which can be  upper bounded via Pseudo-dimension \cite{anthony2009neural}.
	
	\begin{definition}
		Let $\mathcal{F}$ be a class of functions from $X$ to $\mathbb{R}$. Suppose that $S=\{x_1,x_2,\cdots,x_n\}\subset X$. We say that $S$ is pseudo-shattered by $\mathcal{F}$ if there exists $y_1,y_2,\cdots,y_n$ such that for any $b\in\{0,1\}^n$, there exists a $u\in\mathcal{F}$ satisfying
		\begin{equation*}
			\mathrm{sign}(u(x_i)-y_i)=b_i,\quad i=1,2,\dots,n
		\end{equation*}
		and we say that $\{y_i\}_{i=1}^{n}$ witnesses the shattering.
		The pseudo-dimension of $\mathcal{F}$, denoted as $\mathrm{Pdim}(\mathcal{F})$, is defined to be the maximum cardinality among all sets pseudo-shattered by $\mathcal{F}$.
	\end{definition}
	
	The following proposition showing a relation between uniform covering number and pseudo-dimension. 
	\begin{proposition}[Theorem 12.2 \cite{anthony2009neural}]\label{covering number pdim}
		Let $\mathcal{F}$ be a class of real functions from a domain $X$ to the bounded interval $[0, \mathcal{B}]$. Let $\epsilon>0$. Then
		\begin{equation*}
			\mathcal{C}_{\infty}(\epsilon, \mathcal{F}, n) \leq \sum_{i=1}^{\mathrm{Pdim}(\mathcal{F})}\left(\begin{array}{c}
				n \\
				i
			\end{array}\right)\left(\frac{\mathcal{B}}{\epsilon}\right)^{i},
		\end{equation*}
		which is less than $\left(\frac{en\mathcal{B}}{\epsilon\cdot\mathrm{Pdim}(\mathcal{F})}\right)^{\mathrm{Pdim}(\mathcal{F})}$ for $n\geq\mathrm{Pdim}(\mathcal{F})$.
	\end{proposition}
	
\subsection{Bound on Statistical error}
By Lemmas \ref{L decomposition}, \ref{sym}, \ref{chaining} and Proposition \ref{covering number pdim}, we can  bound the  statistical error via
bounding the  pseudo-dimension  of $\mathcal{F}_i, i=1,...,13$. To this end, we  show that $\{\mathcal{F}_i\}$ are subsets of some neural network classes and then bound the pseudo-dimension of associate neural network classes.
	
 In later discussions, a ``$\mathrm{ReLU}^2-\mathrm{ReLU}^3$ network'' is a network with activation functions be either $\mathrm{ReLU}^2$ or $\mathrm{ReLU}^3$, and other terminology such as ``$\mathrm{ReLU}-\mathrm{ReLU}^2-\mathrm{ReLU}^3$'' are defined similarly.
	\begin{proposition} \label{DNN derivative}
		Let $u$ be a function implemented by a $\mathrm{ReLU}^2-\mathrm{ReLU}^3 $  $(\mathrm{ReLU}^3)$ network with depth $\mathcal{D}$ and width $\mathcal{W}$. {Then for $i=1,\cdots,d$, $D_i u:=\frac{\partial u}{\partial x_i}$} can be implemented by a $\mathrm{ReLU}-\mathrm{ReLU}^2-\mathrm{ReLU}^3$ $(\mathrm{ReLU}^2-\mathrm{ReLU}^3)$ network with depth $\mathcal{D}+2$ and width $\left(\mathcal{D}+2\right)\mathcal{W}$.  Moreover, the neural networks implementing $\{D_iu\}_{i=1}^{d}$ have the same architecture.
	\end{proposition}

%
%

	\begin{proof}
		For activation function $\rho$ in each unit, we denote $\widetilde{\rho}$ as its derivative, i.e., $\widetilde{\rho}(x)=\rho'(x)$. We then have
		\begin{equation*}
			\widetilde{\rho}(x)=\left\{
			\begin{array}{ll}
				2\mathrm{ReLU},&\rho=\mathrm{ReLU}^2,\\
				3\mathrm{ReLU}^2,&\rho=\mathrm{ReLU}^3.
			\end{array}
			\right.
		\end{equation*}
		Let $1\leq i\leq d$. We deal with the first two layers in details and apply induction for layers $k\geq3$ since there  are  a  little bit difference for the first two layer. For the first layer, we have for any $q=1,2\cdots,n_1$
		\begin{equation*}
			D_iu_q^{(1)}=D_i\rho_{q}^{(1)}\left(\sum_{j=1}^{d}a_{qj}^{(1)}x_j+b_q^{(1)}\right)
			=\widetilde{\rho_{q}^{(1)}}\left(\sum_{j=1}^{d}a_{qj}^{(1)}x_j+b_q^{(1)}\right)\cdot a_{qi}^{(1)}
		\end{equation*}
		Hence $D_iu_q^{(1)}$ can be implemented by a $\mathrm{ReLU}-\mathrm{ReLU}^2-\mathrm{ReLU}^3$ network with depth $2$ and width $1$. For the second layer,
		\begin{equation*}
			D_iu_q^{(2)}=D_i\rho_q^{(2)}\left(\sum_{j=1}^{n_1}a_{qj}^{(2)}u_j^{(1)}+b_q^{(2)}\right)
			=\widetilde{\rho_q^{(2)}}\left(\sum_{j=1}^{n_1}a_{qj}^{(2)}u_j^{(1)}+b_q^{(2)}\right)\cdot\sum_{j=1}^{n_1}a_{qj}^{(2)}D_iu_j^{(1)}.
		\end{equation*}
		Since $\widetilde{\rho_q^{(2)}}\left(\sum_{j=1}^{n_1}a_{qj}^{(2)}u_j^{(1)}+b_q^{(2)}\right)$ and $\sum_{j=1}^{n_1}a_{qj}^{(2)}D_iu_j^{(1)}$ can be implemented by two $\mathrm{ReLU}-\mathrm{ReLU}^2-\mathrm{ReLU}^3$ subnetworks, respectively, and the multiplication can also be implemented by
		\begin{equation} \label{multiplication by ReLU2}
			\begin{split}
				x\cdot y&=\frac{1}{4}\left[(x+y)^2-(x-y)^2\right]\\
				&=\frac{1}{4}\left[\mathrm{ReLU}^2(x+y)+\mathrm{ReLU}^2(-x-y)-\mathrm{ReLU}^2(x-y)-\mathrm{ReLU}^2(-x+y)\right],
			\end{split}
		\end{equation}
		we conclude that $D_iu_q^{(2)}$ can be implemented by a $\mathrm{ReLU}-\mathrm{ReLU}^2-\mathrm{ReLU}^3$ network. We have $$\mathcal{D}\left(\widetilde{\rho_q^{(2)}}\left(\sum_{j=1}^{n_1}a_{qj}^{(2)}u_j^{(1)}+b_q^{(2)}\right)\right)=3, \mathcal{W}\left(\widetilde{\rho_q^{(2)}}\left(\sum_{j=1}^{n_1}a_{qj}^{(2)}u_j^{(1)}+b_q^{(2)}\right)\right)\leq\mathcal{W}$$ and $$\mathcal{D}\left(\sum_{j=1}^{n_1}a_{qj}^{(2)}D_iu_j^{(1)}\right)=2, \mathcal{W}\left(\sum_{j=1}^{n_1}a_{qj}^{(2)}D_iu_j^{(1)}\right)\leq\mathcal{W}.$$ Thus $\mathcal{D}\left(D_iu_q^{(2)}\right)=4,$ $\mathcal{W}\left(D_iu_q^{(2)}\right)\leq\max\{2\mathcal{W},4\}$.
		Now we apply induction for layers $k\geq3$. For the third layer,
		$
			D_iu_q^{(3)}=D_i\rho_q^{(3)}\left(\sum_{j=1}^{n_2}a_{qj}^{(3)}u_j^{(2)}+b_q^{(3)}\right)
			=\widetilde{\rho_q^{(3)}}\left(\sum_{j=1}^{n_2}a_{qj}^{(3)}u_j^{(2)}+b_q^{(3)}\right)\cdot\sum_{j=1}^{n_2}a_{qj}^{(3)}D_iu_j^{(2)}.
		$
		Since
		$\mathcal{D}\left(\rho_q^{(3)}\left(\sum_{j=1}^{n_2}a_{qj}^{(3)}u_j^{(2)}+b_q^{(3)}\right)\right)=4$, $\mathcal{W}\left(\rho_q^{(3)}\left(\sum_{j=1}^{n_2}a_{qj}^{(3)}u_j^{(2)}+b_q^{(3)}\right)\right)\leq\mathcal{W}$ and $$\mathcal{D}\left(\sum_{j=1}^{n_2}a_{qj}^{(3)}D_iu_j^{(2)}\right)=4, \mathcal{W}\left(\sum_{j=1}^{n_1}a_{qj}^{(3)}D_iu_j^{(2)}\right)\leq\max\{2\mathcal{W},4\mathcal{W}\}=4\mathcal{W},$$ we conclude that $D_iu_q^{(3)}$ can be implemented by a $\mathrm{ReLU}-\mathrm{ReLU}^2-\mathrm{ReLU}^3$ network and $\mathcal{D}\left(D_iu_q^{(3)}\right)=5$, $\mathcal{W}\left(D_iu_q^{(3)}\right)\leq\max\{5\mathcal{W},4\}=5\mathcal{W}$.
		
		We assume that $D_iu_q^{(k)}(q=1,2,\cdots,n_k)$ can be implemented by a $\mathrm{ReLU}$-$\mathrm{ReLU}^2$ network and $\mathcal{D}\left(D_iu_q^{(k)}\right)=k+2$, $\mathcal{W}\left(D_iu_q^{(3)}\right)\leq(k+2)\mathcal{W}$. For the $(k+1)-$th layer,
		$
			D_iu_q^{(k+1)}=D_i\rho_q^{(k+1)}\left(\sum_{j=1}^{n_k}a_{qj}^{(k+1)}u_j^{(k)}+b_q^{(k+1)}\right) =\widetilde{\rho_q^{(k+1)}}\left(\sum_{j=1}^{n_k}a_{qj}^{(k+1)}u_j^{(k)}+b_q^{(k+1)}\right)\cdot\sum_{j=1}^{n_k}a_{qj}^{(k+1)}D_iu_j^{(k)}.$
		Since  $\mathcal{D}\left(\widetilde{\rho_q^{(k+1)}}\left(\sum_{j=1}^{n_k}a_{qj}^{(k+1)}u_j^{(k)}+b_q^{(k+1)}\right)\right)=k+2$, $\mathcal{W}\left(\widetilde{\rho_q^{(k+1)}}\left(\sum_{j=1}^{n_k}a_{qj}^{(k+1)}u_j^{(k)}+b_q^{(k+1)}\right)\right)\leq\mathcal{W}$ and $\mathcal{D}\left(\sum_{j=1}^{n_k}a_{qj}^{(k+1)}D_iu_j^{(k)}\right)=k+2$, $\mathcal{W}\left(\sum_{j=1}^{n_k}a_{qj}^{(k+1)}D_iu_j^{(k)}\right)\leq\max\{(k+2)\mathcal{W},4\mathcal{W}\}=(k+2)\mathcal{W}$, we conclude that $D_iu_q^{(k+1)}$ can be implemented by a $\mathrm{ReLU}-\mathrm{ReLU}^2-\mathrm{ReLU}^3$ network and $\mathcal{D}\left(D_iu_q^{(k+1)}\right)=k+3$, $\mathcal{W}\left(D_iu_q^{(k+1)}\right)\leq\max\{(k+3)\mathcal{W},4\}=(k+3)\mathcal{W}$.
		Hence we derive that $D_iu=D_iu_1^{\mathcal{D}}$ can be implemented by a $\mathrm{ReLU}-\mathrm{ReLU}^2-\mathrm{ReLU}^3$ network and $\mathcal{D}\left(D_iu\right)=\mathcal{D}+2$, $\mathcal{W}\left(D_iu\right)\leq\left(\mathcal{D}+2\right)\mathcal{W}$. And through our argument, we know that the neural networks implementing $\{D_iu\}_{i=1}^{d}$ have the same architecture.
	\end{proof}
	
	We now present the bound for pseudo-dimension of $\mathcal{N}(\mathcal{D},\mathcal{W},\{\mathrm{ReLU},\mathrm{ReLU}^2,\mathrm{ReLU}^3\})$ with $\mathcal{D},\mathcal{W}\in\mathbb{N}$. We  need the following  Lemma.
	\begin{lemma}(Theorem 8.3 in \cite{anthony2009neural}) \label{lemma covering number}
		Let $p_1,\cdots,p_m$ be polynomials with $n$ variables  of degree at most $d$. If $n\leq m$, then
		\begin{equation*}
			|\{(\mathrm{sign}(p_1(x)),\cdots,\mathrm{sign}(p_m(x))):x\in\mathbb{R}^n\}|\leq2\left(\frac{2emd}{n}\right)^n.
		\end{equation*}
	\end{lemma}

	\begin{proposition}\label{pdimb}
		For any $\mathcal{D},\mathcal{W}\in\mathbb{N}$,
		\begin{equation*}
			\mathrm{Pdim}(\mathcal{N}(\mathcal{D},\mathcal{W},\{\mathrm{ReLU},\mathrm{ReLU}^2,\mathrm{ReLU}^3\}))=\mathcal{O}(\mathcal{D}^2\mathcal{W}^2(\mathcal{D}+\log\mathcal{W})).
		\end{equation*}	
	\end{proposition}
	

	\begin{proof}
		The argument is follows from the  proof of Theorem 6 in \cite{bartlett2019nearly}.  The result stated here is somewhat stronger then Theorem 6 in \cite{bartlett2019nearly} since $\mathrm{VCdim}(\mathrm{sign}(\mathcal{F}))\leq \mathrm{Pdim}(\mathcal{F})$ for any function class $\mathcal{F}$. We consider a new set of functions
		$$
			\mathcal{\widetilde{N}}=\{\widetilde{u}(x,y)=\mathrm{sign}(u(x)-y): u\in\mathcal{N}(\mathcal{D},\mathcal{W},\{\mathrm{ReLU},\mathrm{ReLU}^2,\mathrm{ReLU}^3\})\}.$$
		It is clear that $\mathrm{Pdim}(\mathcal{N}(\mathcal{D},\mathcal{W},\{\mathrm{ReLU},\mathrm{ReLU}^2,\mathrm{ReLU}^3\}))\leq\mathrm{VCdim}(\mathcal{\widetilde{N}})$. We now bound the VC-dimension of $\mathcal{\widetilde{N}}$. Denoting $\mathcal{M}$ as the total number of parameters (weights and biases) in the neural network implementing functions in $\mathcal{N}$, in our case we want to derive the uniform bound for
		\begin{equation*}
			K_{\{x_i\},\{y_i\}}(m):=|\{(\operatorname{sign}(f(x_{1}, a)-y_1), \ldots, \operatorname{sign}(u(x_{m}, a)-y_m)): a \in \mathbb{R}^{\mathcal{M}}\}|
		\end{equation*}
		over all $\{x_i\}_{i=1}^{m}\subset X$ and $\{y_i\}_{i=1}^{m}\subset\mathbb{R}$. Actually the maximum of $K_{\{x_i\},\{y_i\}}(m)$ over all $\{x_i\}_{i=1}^{m}\subset X$ and $\{y_i\}_{i=1}^{m}\subset\mathbb{R}$ is the growth function $\mathcal{G}_{\mathcal{\widetilde{N}}}(m)$.
		In order to apply Lemma \ref{lemma covering number}, we partition the parameter space $\mathbb{R}^{\mathcal{M}}$ into several subsets to ensure that in each subset $u(x_i,a)-y_i$ is a polynomial with respect to $a$ without any breakpoints. In fact, our partition is exactly the same as the partition in \cite{bartlett2019nearly}. Denote the partition as $\{P_1,P_2,\cdots,P_N\}$ with some integer $N$ satisfying
		\begin{equation} \label{pdimb1}
			N\leq\prod_{i=1}^{\mathcal{D}-1}2\left(\frac{2emk_i(1+(i-1)3^{i-1})}{\mathcal{M}_i}\right)^{\mathcal{M}_i}
		\end{equation}
		where $k_i$ and $\mathcal{M}_i$ denotes the number of units at the $i$th layer and the total number of parameters at the inputs to units in all the layers up to layer $i$ of the neural network implementing functions in $\mathcal{N}$, respectively. See \cite{bartlett2019nearly} for the construction of the partition. Obviously we have
		\begin{equation} \label{pdimb2}
			K_{\{x_i\},\{y_i\}}(m)\leq\sum_{i=1}^{N}|\{(\operatorname{sign}(u(x_{1}, a)-y_1), \cdots, \operatorname{sign}(u(x_{m}, a)-y_m)): a\in P_i\}|
		\end{equation}
		Note that $u(x_i,a)-y_i$ is a polynomial with respect to $a$ with degree the same as the degree of $u(x_i,a)$, which is equal to $1 + (\mathcal{D}-1)3^{\mathcal{D}-1}$ as shown in \cite{bartlett2019nearly}.  Hence by  Lemma \ref{lemma covering number}, we have
		\begin{equation}
		   \begin{aligned}
			&|\{(\operatorname{sign}(u(x_{1}, a)-y_1), \cdots, \operatorname{sign}(u(x_{m}, a)-y_m)): a\in P_i\}|\\
	&\leq2\left(\frac{2em(1+(\mathcal{D}-1)3^{\mathcal{D}-1})}{\mathcal{M}_{\mathcal{D}}}\right)^{\mathcal{M}_{\mathcal{D}}}\label{pdimb3}.
		   \end{aligned}
		\end{equation}
		Combining (\ref{pdimb1}), (\ref{pdimb2}) and  (\ref{pdimb3}) yields
		\begin{equation*}
			K_{\{x_i\},\{y_i\}}(m)\leq\prod_{i=1}^{\mathcal{D}}2\left(\frac{2emk_i(1+(i-1)3^{i-1})}{\mathcal{M}_i}\right)^{\mathcal{M}_i}.
		\end{equation*}
		We then have
		\begin{equation*}
			\mathcal{G}_{\mathcal{\widetilde{N}}}(m)\leq\prod_{i=1}^{\mathcal{D}}2\left(\frac{2emk_i(1+(i-1)3^{i-1})}{\mathcal{M}_i}\right)^{\mathcal{M}_i},
		\end{equation*}
		since the maximum of $K_{\{x_i\},\{y_i\}}(m)$ over all $\{x_i\}_{i=1}^{m}\subset X$ and $\{y_i\}_{i=1}^{m}\subset\mathbb{R}$ is the growth function $\mathcal{G}_{\mathcal{\widetilde{N}}}(m)$. Doing some algebras  as that of  the proof of Theorem 6 in \cite{bartlett2019nearly}, we obtain
		\begin{equation*}
		   \begin{aligned}
			\mathrm{Pdim}(\mathcal{N}(\mathcal{D},\mathcal{W},\{\mathrm{ReLU},\mathrm{ReLU}^2,\mathrm{ReLU}^3\}))&\leq \mathcal{O}\left(\mathcal{D}^2\mathcal{W}^2\left( \mathcal{D}+\log \mathcal{W}\right)\right).
		   \end{aligned}
		\end{equation*}
	\end{proof}

	With the above preparations,  we are able to derive our result on  the statistical error.
	\begin{theorem} \label{sta error}
		Let $\mathcal{D},\mathcal{W}\in\mathbb{N},\mathcal{B}\in\mathbb{R}^+$. For any $\epsilon>0$, if the number of sample
		\begin{equation*}
			N,M= C(d,\Omega,\mathfrak{Z},\mathcal{B})\mathcal{D}^6\mathcal{W}^2(\mathcal{D}+\log\mathcal{W})\left(\frac{1}{\epsilon}\right)^{2+\delta}
		\end{equation*}
		where $\delta$ is an arbitrarily small number
then we have
		\begin{equation*}
			\mathbb{E}_{\{{X_k}\}_{k=1}^{N},\{{Y_k}\}_{k=1}^{M}}\sup_{u\in\mathcal{P}}\left|\mathcal{L}(u)-\widehat{\mathcal{L}}(u)\right|\leq \epsilon,
		\end{equation*}
		where $\mathcal{P}=\mathcal{N}(\mathcal{D},\mathcal{W},\{\|\cdot\|_{C^2(\bar{\Omega})},\mathcal{B}\},\{\mathrm{ReLU}^3\})$.
	\end{theorem}
	
\begin{proof}
We need the following Lemma. Let $\Phi = \{\mathrm{ReLU},\mathrm{ReLU}^2,\mathrm{ReLU}^3\}$.

\begin{lemma} \label{F subset N}
{Let $\{\mathcal{F}_i\}_{i=1}^{13}$ be defined in Lemma \ref{sym}. There holds}
\begin{equation*}
\begin{aligned}
			\mathcal{F}_1&\subset\mathcal{N}_1:=\mathcal{N}(\mathcal{D}+5,2(\mathcal{D}+2)(\mathcal{D}+4)\mathcal{W},\{\|\cdot\|_{C(\bar{\Omega})},\mathcal{B}^2\},\Phi)\\
			\mathcal{F}_2&\subset\mathcal{N}_2:=\mathcal{N}(\mathcal{D}+3,2(\mathcal{D}+2)\mathcal{W},\{\|\cdot\|_{C(\bar{\Omega})},\mathcal{B}^2\},\Phi)\\
			\mathcal{F}_3&\subset\mathcal{N}_3:=\mathcal{N}(\mathcal{D}+1,\mathcal{W},\{\|\cdot\|_{C(\bar{\Omega})},\mathcal{B}^2\},\Phi)\\
			\mathcal{F}_5&\subset\mathcal{N}_5:=\mathcal{N}(\mathcal{D}+5,(\mathcal{D}+2)(\mathcal{D}+5)\mathcal{W},\{\|\cdot\|_{C(\bar{\Omega})},\mathcal{B}^2\},\Phi)\\
			\mathcal{F}_6&\subset\mathcal{N}_6:=\mathcal{N}(\mathcal{D}+3,(\mathcal{D}+3)\mathcal{W},\{\|\cdot\|_{C(\bar{\Omega})},\mathcal{B}^2\},\Phi)\\
			\mathcal{F}_7&\subset\mathcal{N}_7:=\mathcal{N}(\mathcal{D}+4,(\mathcal{D}+2)(\mathcal{D}+4)\mathcal{W},\{\|\cdot\|_{C(\bar{\Omega})},\mathcal{B}\},\Phi)\\
			\mathcal{F}_8&\subset\mathcal{N}_8:=\mathcal{N}(\mathcal{D}+3,(\mathcal{D}+3)\mathcal{W},\{\|\cdot\|_{C(\bar{\Omega})},\mathcal{B}^2\},\Phi)\\
			\mathcal{F}_9&\subset\mathcal{N}_9:=\mathcal{N}(\mathcal{D}+2,(\mathcal{D}+2)\mathcal{W},\{\|\cdot\|_{C(\bar{\Omega})},\mathcal{B}\},\Phi)\\
			\mathcal{F}_{10}&\subset\mathcal{N}_{10}:=\mathcal{N}(\mathcal{D},\mathcal{W},\{\|\cdot\|_{C(\bar{\Omega})},\mathcal{B}\},\Phi)\\
			\mathcal{F}_{11}&\subset\mathcal{N}_{11}:=\mathcal{N}(\mathcal{D}+1,\mathcal{W},\{\|\cdot\|_{C(\bar{\Omega})},\mathcal{B}^2\},\Phi)\\
			\mathcal{F}_{13}&\subset\mathcal{N}_{13}:=\mathcal{N}(\mathcal{D},\mathcal{W},\{\|\cdot\|_{C(\bar{\Omega})},\mathcal{B}\},\Phi).
\end{aligned}
\end{equation*}
	\end{lemma}
	\begin{proof}
		By Proposition \ref{DNN derivative}, we know that for $u$ is a $\mathrm{ReLU}^3$ network with depth $\mathcal{D}$ and width $\mathcal{W}$, $u_{x_i}$ can be implemented by a $\mathrm{ReLU}^2-\mathrm{ReLU}^3$ network with depth $\mathcal{D}+2$ and width $(\mathcal{D}+2)\mathcal{W}$. Then by Proposition \ref{DNN derivative} again we have that $u_{x_ix_j}$ can be implemented by a $\mathrm{ReLU}-\mathrm{ReLU}^2-\mathrm{ReLU}^3$ network with depth $\mathcal{D}+4$ and width $(\mathcal{D}+2)(\mathcal{D}+4)\mathcal{W}$. These facts combining with (\ref{multiplication by ReLU2}) yields the results.
	\end{proof}

	By Lemma  \ref{chaining} and  Proposition  \ref{covering number pdim}, we have for $i=1,2,...,10$,
		\begin{equation} \label{sta error1}
			\begin{split}
				\mathfrak{R}_{{U(\Omega),N}}(\mathcal{F}_i)&\leq\inf_{0<\delta<\mathcal{B}_i}\left(4 \delta+\frac{12}{\sqrt{N}} \int_{\delta}^{\mathcal{B}_i} \sqrt{\log (2\mathcal{C}\left( \epsilon, \mathcal{F}_i, N\right))}\mathrm{d}\epsilon\right)\\
				&\leq\inf_{0<\delta<\mathcal{B}_i}\left(4 \delta+\frac{12}{\sqrt{N}} \int_{\delta}^{\mathcal{B}_i} \sqrt{\log \left(2\left(\frac{eN\mathcal{B}_i}{\epsilon\cdot\mathrm{Pdim}(\mathcal{F}_i)}\right)^{\mathrm{Pdim}(\mathcal{F}_i)}\right)}\mathrm{d}\epsilon\right)\\
				&\leq\inf_{0<\delta<\mathcal{B}_i}\left(4 \delta+\frac{12\mathcal{B}_i}{\sqrt{N}}+\frac{12}{\sqrt{N}} \int_{\delta}^{\mathcal{B}_i} \sqrt{{\mathrm{Pdim}(\mathcal{F}_1)}\cdot\log \left(\frac{eN\mathcal{B}_i}{\epsilon\cdot\mathrm{Pdim}(\mathcal{F}_i)}\right)}\mathrm{d}\epsilon\right).
			\end{split}
		\end{equation}
		Now we calculate the integral. Let
		$
			t=\sqrt{\log(\frac{eN\mathcal{B}_i}{\epsilon\cdot\text{Pdim}(\mathcal{F}_i)})},
		$
		then $\epsilon=\frac{eN\mathcal{B}_i}{\mathrm{Pdim}(\mathcal{F}_i)}\cdot e^{-t^2}$. Denoting $t_1=\sqrt{\log(\frac{eN\mathcal{B}_i}{\mathcal{B}_i\cdot\mathrm{Pdim}(\mathcal{F}_i)})}$, $t_2=\sqrt{\log(\frac{eN\mathcal{B}_i}{\delta\cdot\mathrm{Pdim}(\mathcal{F}_i)})}$, we have
\begin{align} \int_{\delta}^{\mathcal{B}_i}\sqrt{\log\left(\frac{eN\mathcal{B}_i}{\epsilon\cdot\mathrm{Pdim}(\mathcal{F}_i)}\right)}d\epsilon&=\frac{2eN\mathcal{B}_i}{\mathrm{Pdim}(\mathcal{F}_i)}\int_{t_1}^{t_2}t^2e^{-t^2}dt\nonumber\\
				&=\frac{2eN\mathcal{B}_i}{\mathrm{Pdim}(\mathcal{F}_i)}\int_{t_1}^{t_2}t\left(\frac{-e^{-t^2}}{2}\right)'dt\nonumber\\
				&=\frac{eN\mathcal{B}_i}{\mathrm{Pdim}(\mathcal{F}_i)}\left[t_1e^{-t_1^2}-t_2e^{-t_2^2}+\int_{t_1}^{t_2}e^{-t^2}dt\right]\nonumber\\
				&\leq\frac{eN\mathcal{B}_i}{\mathrm{Pdim}(\mathcal{F}_i)}\left[t_1e^{-t_1^2}-t_2e^{-t_2^2}+(t_2-t_1)e^{-t_1^2}\right]\nonumber\\
				&\leq\frac{eN\mathcal{B}_i}{\mathrm{Pdim}(\mathcal{F}_i)}\cdot t_2e^{-t_1^2}=\|u\|\sqrt{\log\left(\frac{eN\mathcal{B}_i}{\delta\cdot\mathrm{Pdim}(\mathcal{F}_i)}\right)}\label{sta error2}
\end{align}
		Combining (\ref{sta error1}) and (\ref{sta error2}) and choosing $\delta=\mathcal{B}_i\left(\frac{\mathrm{Pdim}(\mathcal{F}_i)}{N}\right)^{1/2}\leq\mathcal{B}_i$, we get for  $i=1$,$2$,$3$,$5$,$6$,$7$,$8$,$9$,$10$,
		\begin{equation} \label{sta error3}
			\begin{split}
				\mathfrak{R}_{{U(\Omega),N}}(\mathcal{F}_i)&
				\leq\inf_{0<\delta<\mathcal{B}_i}\left(4 \delta+\frac{12\mathcal{B}_i}{\sqrt{N}}+\frac{12}{\sqrt{N}} \int_{\delta}^{\mathcal{B}_i} \sqrt{\mathrm{Pdim}(\mathcal{F}_i)\cdot\log \left(\frac{eN\mathcal{B}_i}{\epsilon\mathrm{Pdim}(\mathcal{F}_i)}\right)}\mathrm{d}\epsilon\right)\\
				&\leq \inf_{0<\delta<\mathcal{B}_i}\left(4\delta+\frac{12\mathcal{B}_i}{\sqrt{N}}+\frac{12\mathcal{B}_i\sqrt{\mathrm{Pdim}(\mathcal{F}_i)}}{\sqrt{N}}\sqrt{\log\left(\frac{eN\mathcal{B}_i}{\delta\cdot\mathrm{Pdim}(\mathcal{F}_i)}\right)}\right)\\
				&\leq 28\sqrt{\frac{3}{2}}\mathcal{B}_i\left(\frac{\mathrm{Pdim}(\mathcal{F}_i)}{N}\right)^{1/2}\sqrt{\log\left(\frac{eN}{\mathrm{Pdim}(\mathcal{F}_i)}\right)}\\
				&\leq 28\sqrt{\frac{3}{2}}\mathcal{B}_i\left(\frac{\mathrm{Pdim}(\mathcal{N}_i)}{N}\right)^{1/2}\sqrt{\log\left(\frac{eN}{\mathrm{Pdim}(\mathcal{N}_i)}\right)}\\	
				&\leq 28\sqrt{\frac{3}{2}}\max\{\mathcal{B},\mathcal{B}^2\}\left(\frac{\mathcal{H}}{N}\right)^{1/2}\sqrt{\log\left(\frac{eN}{\mathcal{H}}\right)},
			\end{split}
		\end{equation}
		with
		\begin{equation*}	\mathcal{H}=4C(\mathcal{D}+2)^2(\mathcal{D}+4)^2(\mathcal{D}+5)^2\mathcal{W}^2(\mathcal{D}+5+\log((\mathcal{D}+2)(\mathcal{D}+4)\mathcal{W})).
		\end{equation*}
		where in the forth step we apply Lemma \ref{F subset N} and we use Proposition \ref{pdimb} in the last step. Similarly for $i=11,13$,
		\begin{equation} \label{sta error4}
			\mathfrak{R}_{{U(\Omega),N}}(\mathcal{F}_i)
			\leq 28\sqrt{\frac{3}{2}}\max\{\mathcal{B},\mathcal{B}^2\}\left(\frac{\mathcal{H}}{M}\right)^{1/2}\sqrt{\log\left(\frac{eM}{\mathcal{H}}\right)}
		\end{equation}
		Obviously, $\mathfrak{R}_{{U(\Omega),N}}(\mathcal{F}_4)$ and $\mathfrak{R}_{{U(\Omega),N}}(\mathcal{F}_{12})$ can be bounded by the right hand side of (\ref{sta error3}) and (\ref{sta error4}), respectively. Combining Lemma \ref{L decomposition} and  \ref{sym} and (\ref{sta error3}) and (\ref{sta error4}), we have
		{\small\begin{equation*}
			\begin{split}
			&\mathbb{E}_{\{{X_k}\}_{k=1}^{N},\{{Y_k}\}_{k=1}^{M}}\sup_{u\in\mathcal{P}}\left|\mathcal{L}(u)-\widehat{\mathcal{L}}(u)\right|	\leq\sum_{j=1}^{13}\mathbb{E}_{\{{X_k}\}_{k=1}^{N},\{{Y_k}\}_{k=1}^{M}}\sup_{u\in\mathcal{P}}\left|\mathcal{L}_j(u)-\widehat{\mathcal{L}_j}(u)\right|\\ &\leq28\sqrt{\frac{3}{2}}\max\{\mathcal{B},\mathcal{B}^2\}\left(40|\Omega|C_1\left(\frac{\mathcal{H}}{N}\right)^{1/2}\sqrt{\log\left(\frac{eN}{\mathcal{H}}\right)}+12|\partial\Omega|C_2\left(\frac{\mathcal{H}}{M}\right)^{1/2}\sqrt{\log\left(\frac{eM}{\mathcal{H}}\right)}\right),
			\end{split}
		\end{equation*}}
		where
		\begin{equation*}
			\begin{split} C_1&=\max\{\mathfrak{A}^2d^4,\mathfrak{B}^2d^2,\mathfrak{C}^2d^2,\mathfrak{F}^2,\mathfrak{A}\mathfrak{B}d^3,\mathfrak{A}\mathfrak{C}d^2,\mathfrak{A}\mathfrak{F}d^2,\mathfrak{B}\mathfrak{C}d,\mathfrak{B}\mathfrak{F}d,\mathfrak{C}\mathfrak{F}\},\\
				C_2&=\max\{\mathfrak{E}^2,\mathfrak{G}^2,\mathfrak{E}\mathfrak{G}\}.
			\end{split}
		\end{equation*}
		Hence for any $\epsilon>0$, if the number of sample
		\begin{equation*}
			N,M= C(d,\Omega,\mathfrak{Z},\mathcal{B})\mathcal{D}^6\mathcal{W}^2(\mathcal{D}+\log\mathcal{W})\left(\frac{1}{\epsilon}\right)^{2+\delta}
		\end{equation*}
		with $\delta$ being an arbitrarily small number, then we have
		\begin{equation*}
		\mathbb{E}_{\{{X_k}\}_{k=1}^{N},\{{Y_k}\}_{k=1}^{M}}\sup_{u\in\mathcal{P}}\left|\mathcal{L}(u)-\widehat{\mathcal{L}}(u)\right|\leq \epsilon.
		\end{equation*}
	\end{proof}

	\section{Convergence rate   for the PINNs}\label{section convergence rate for the pinns}
With the preparation in last two sections on the bounds of approximation and statistical errors, we will give the main results in this section.
	\begin{theorem} \label{total error}
		Let Assumption \ref{ass:exist} holds true and further assume $u^* \in C^3(\bar\Omega)$. For any $\epsilon>0$, we choose the parameterized neural network  class
		\begin{equation*}	\mathcal{P}=\mathcal{N}\left(\lceil\log_2d\rceil+2,C(d,\Omega,\mathfrak{Z},\|u^*\|_{C^3(\bar{\Omega})})\left(\frac{1}{\epsilon}\right)^{d},\{\|\cdot\|_{C^2(\bar{\Omega})},2\|u^*\|_{C^2(\bar{\Omega})}\},\{\mathrm{ReLU}^3\}\right)
		\end{equation*}
		and let the number of samples be
		\begin{equation*}
			N,M=C(d,\Omega,\mathfrak{Z},\|u^*\|_{C^3(\bar{\Omega})})\left(\frac{1}{\epsilon}\right)^{2d+4+\delta}
		\end{equation*}
		where $\delta$ is an arbitrarily positive number,
  then we have
		\begin{equation*}
			\mathbb{E}_{\{{X_k}\}_{k=1}^{N},\{{Y_k}\}_{k=1}^{M}}\left\|u_{\phi}-u^{*}\right\|_{H^{\frac{1}{2}}(\Omega)}\leq\epsilon.
		\end{equation*}
	\end{theorem}
	\begin{proof}
		For any $\epsilon>0$, by Theorem \ref{app error}, there exists an neural network function $\bar{u}$ with depth $\lceil\log_2d\rceil+2$ and width $C(d,\Omega,\mathfrak{Z},\|u^*\|_{C^3(\bar{\Omega})})\left(\frac{1}{\epsilon}\right)^{d}$ such that
		\begin{equation*}
			\|u^*-\bar{u}\|_{C^2(\bar{\Omega})}\leq \left(\frac{\epsilon^2}{3C(d^2+3d+4)|\Omega|\max\{2d^2\mathfrak{A}^2,d\mathfrak{B}^2,\mathfrak{C}^2\}+2C|\partial\Omega|\mathfrak{E}^2}\right)^{1/2}.
		\end{equation*}
		Without loss of generality we assume that $\epsilon$ is small enough such that
		\begin{equation*}
			\|\bar{u}\|_{C^2(\bar{\Omega})}\leq\|u^*-\bar{u}\|_{C^2(\bar{\Omega})}+\|u^*\|_{C^2(\bar{\Omega})}\leq 2\|u^*\|_{C^2(\bar{\Omega})}.
		\end{equation*}
		Hence $\bar{u}$ belongs to the function class
		\begin{equation*} \mathcal{P}=\mathcal{N}\left(\lceil\log_2d\rceil+2,C(d,\Omega,\mathfrak{Z},\|u^*\|_{C^3(\bar{\Omega})})\left(\frac{1}{\epsilon}\right)^{d},\{\|\cdot\|_{C^2(\bar{\Omega})},2\|u^*\|_{C^2(\bar{\Omega})}\},\{\mathrm{ReLU}^3\}\right).
		\end{equation*}
		And
		\begin{equation} \mathcal{E}_{app}\leq\frac{3}{2}(d^2+3d+4)|\Omega|\max\{2d^2\mathfrak{A}^2,d\mathfrak{B}^2,\mathfrak{C}^2\}\|\bar{u}-u^*\|_{C^2(\bar{\Omega})}^2+|\partial\Omega|\mathfrak{E}^2\|\bar{u}-u^*\|_{C^2(\bar{\Omega})}^2\leq\frac{\epsilon^2}{2C}\label{error1}.
		\end{equation}
		By Theorem \ref{sta error}, when the number of sample
		\begin{equation*}
			N,M= C(d,\Omega,\mathfrak{Z},\mathcal{B})\mathcal{D}^6\mathcal{W}^2(\mathcal{D}+\log\mathcal{W})\left(\frac{1}{\epsilon}\right)^{4+\delta}=C(d,\Omega,\mathfrak{Z},\|u^*\|_{C^3(\bar{\Omega})})\left(\frac{1}{\epsilon}\right)^{2d+4+\delta}
		\end{equation*}
		with $\delta$ being an arbitrarily positive number,  we have
		\begin{equation} \label{error2} \mathcal{E}_{sta}=\mathbb{E}_{\{{X_k}\}_{k=1}^{N},\{{Y_k}\}_{k=1}^{M}}\sup_{u\in\mathcal{P}}\left|\mathcal{L}(u)-\widehat{\mathcal{L}}(u)\right|\leq \frac{\epsilon^2}{2C}
		\end{equation}
		Combining Proposition \ref{error decomposition},  (\ref{error1}) and (\ref{error2}) yields the result.
	\end{proof}

\begin{remark}
The asymptotic convergence results for PINNs have been studied in \cite{shin2020convergence,shin2020error,mishra2020}, i.e., when the number of parameters in the neural networks and number of training samples go to infinity,  the solution of PINNs with converges to the solution of the PDEs'. Our work establishes a nonasymptotic convergence rate of PINNs.
According to our results in Theorem \ref{total error},    the influence  of  the depth and width in the neural networks and number of training samples are characterized quantitatively.  It gives an answer on how to choose the hyperparameters to archive the desired accuracy, which is missing in \cite{shin2020convergence,shin2020error,mishra2020}.
\end{remark}
\begin{remark}
As shown in Theorem \ref{total error},   the PINNs suffers from the curse of dimensionality.
However, if no additional smoothness or low-dimensional compositional structure on the underlying  solutions of  PDEs are imposed, the curse is unavoidable. Recently,   the  minimax lower bound for  PINNs  are proved in \cite{luoptimal}, where they showed that to achieve error of $\mathcal{\epsilon}$ for PINNs to solve  elliptical equation defined  on $[0,1]^{d}$ whose solution lives in $H^1([0,1]^{d})$, the number of samples $n$ are lower bounded by $\mathcal{O}((1/\epsilon)^d).$ In \cite{lu2021priori}, the authors reduce the curses by assuming the underlying solutions living in  Barron spaces which is much smaller than $H^1$. One can also utilize the structures of the solution to reduce the curse, for example by considering PDEs whose solution are composition of functions with number of variables much smaller than $d$.
\end{remark}

\section{Conclusion}\label{conclusion}
\par This paper provided an analysis of convergence rate for PINNs.  Our results give a way about how to set depth and width of networks to achieve the desired convergence rate in terms of number of training samples. The estimation on the approximation error of deep $\mathrm{ReLU}^{3}$ network is established in $C^2$ norms. The statistical error can be derived technically by the Rademacher complexity of the non-Lipschitz composition with  $\mathrm{ReLU}^{3}$ network. It is interesting to extend the current analysis of PINNs to PDEs with other boundary conditions or the optimal control (inverse) problems.

\section*{Acknowledgments}\label{section acknowledgments}
The authors would like to thank the anonymous referees for their useful suggestions that help
improve the manuscript.

\bibliographystyle{siam}
\bibliography{ref}

\end{document}